\documentclass[a4paper, english, 10pt]{amsart}

\usepackage{amsthm,amssymb,url,hyperref}
\usepackage{fullpage}
\usepackage{stmaryrd}
\usepackage[all]{xy}
\usepackage{color}
\usepackage[all]{xy}
\usepackage{tikz}
\usepackage{color,tikz}

\theoremstyle{plain}
\newtheorem{theorem}{Theorem}[section]

\newtheorem{corollary}[theorem]{Corollary}

\newtheorem{lemma}[theorem]{Lemma}

\newtheorem{proposition}[theorem]{Proposition}

\theoremstyle{definition}

\newtheorem{remark}[theorem]{Remark}

\newtheorem*{remark*}{Remark}

\def\!#1{#1^\s}

\def\ker#1{\mathrm{ker}(#1)}

\def\aut#1{\mathrm{Aut}(#1)}

\def\End#1{\mathrm{End}(#1)}
\def\aff#1{\mathrm{Aff}#1}

\def\lmlt{\mathrm{LMlt}}
\def\dis{\mathrm{Dis}}
\def\sym{\mathrm{Sym}}

\newcommand{\BlocS}[2]{\dis(#1)_{[#2]}}

\def\s{\mathfrak{s}}

\def\comment#1{{\color{red} #1}}

\def\setof#1#2{\{#1\, : \,#2\}}

\def\N{\mathrm{Norm}}
\def\c#1{\mathrm{con}_{#1}}

\def\cg#1{\equiv_\alpha}
\newcommand*\xbar[1]{%
   \hbox{%
     \vbox{%
       \hrule height 0.5pt 
       \kern0.5ex
       \hbox{%
         \kern-0.1em
         \ensuremath{#1}%
         \kern-0.1em
       }%
     }%
   }%
}

\usepackage{MnSymbol}

\title{Nilpotent left quasigroups}
\author{Marco Bonatto}

\address[M. Bonatto]{Dipartimento di matematica e informatica - UNIFE}

\email{marco.bonatto.87@gmail.com}

\begin{document}

\begin{abstract}
In this paper we investigate central congruence of left quasigroups in the sense of Freese and McKenzie \cite{comm} and we extend some known results for quandles. In particular, we can extend the characterization of finite nilpotent latin quandles and the characterization of distributive varieties of quandles to the setting of idempotent left quasigroups.
\end{abstract}

\maketitle
\section{Introduction}

%

Left quasigroups are binary algebraic structures with a very combinatorial flavour. Many algebraic structures of interest have an underlying left quasigroup structure (e.g. racks and quandles \cite{AG,J, Matveev}, cycle sets \cite{Rump1} etc) and therefore it is worth to study such structures as a common ground. In this paper we basically show that many results stated for quandles in \cite{Principal, Super, Maltsev_paper} can be extendend to (idempotent) left quasigroups with a focus on {\it nilpotent} left quasigroups in the sense of the commutator theory developed in \cite{Smith, comm}. This theory generalized the well-known concepts of {\it abeliannes, solvability} and {\it nilpotency} in group theory to arbitrary algebraic structures and it has been specialized to the setting of racks and quandles in \cite{CP}.


One of the goals of the paper is to remark the interplay between the lattice of congruences and the lattice of {\it admissible subgroups} defined in \cite{LTT}. The two lattices are related by a Galois connection (see \cite[Theorem 1.10]{LTT}) that can be used to transfer information from one lattice to another. 

For racks, the admissible subgroups encode many information, including also the property of abelianness and centrality of congruences. In particular, such properties of congruences are reflected by the properties of the {\it relative displacement groups} \cite[Lemma 5.1, Proposition 5.2]{CP}. The structure of the displacement group is also effected by the properties of congruences: for instance a rack is solvable (resp. nilpotent) if and only if its displacement group is solvable (resp. nilpotent) \cite[Lemma 6.1, Lemma 6.2]{CP}. Moreover, we have a prime-power decomposition theorem for quandle \cite[Section 6.2]{CP} and in particular connected racks of prime power order are nilpotent. For left quasigroups the relation between congruences and admissible subgroups is not so tight, nevertheless we show that some weaker results hold in this direction as Corollary \ref{solvable} and Lemma \ref{divisors}.

Central congruences and {\it semiregular} admissible subgroups are closely related. In particular, some properties of central congruences are actually determined by the property of being semiregular for the action of the displacement group on every block of the congruence. Thus, we investigate congruences obtained from semiregular admissible subgroups through the Galois connection (under some further mild assumptions every central congruence arise in this way) and semiregular left quasigroups. These property are determined by a family of the equivalence relations already introduced for semimedial left quasigroups in \cite{LTT}. We also obtain that semiregular quandles are the building blocks of idempotent left quasigroups (see Proposition \ref{semiregular decomposition}). Moreover, abelian left quasigroups (in the sense of \cite{comm}) are semiregular and the idempotent ones are actually quandles (see Corollary \ref{corollary for abelian}(ii)).

The results collected in the first part of the paper can be used to investigate Malt'sev varieties of idempotent left quasigroups. In particular, Theorem \ref{nilpotent are latin}, Proposition \ref{nilpotent latin}, Proposition \ref{maltsev semiregular} and Corollary \ref{distributive} are direct generalization of known results for quandles \cite{Maltsev_paper}.

The paper is organized as follows: in Section \ref{preliminary} we recall some basic facts about left quasigroups, congruences and admissible subgroups and the {\it Cayley kernel}. In Section \ref{Semiregular section} we define and study semiregular left quasigroups and semiregular admissible subgroups. In Section \ref{nilpotent section} we recall some basic facts about commutator theory and we explore central congruences and central extensions of left quasigroups. In the last Section we turn out attention to Malt'sev varieties of left quasigroups and we extend some results from \cite{Maltsev_paper}.

\subsection{Notation and terminology}
An {\it algebraic structure} is given by a set $A$ with an arbitrary set of basic operations $F$. A {\it term} on $A$ is either a variable or an expression $f(t_1,\ldots,t_{n})$, where $t_1,\ldots,t_{n}$ are terms on $A$ and $f$ is a basic $n$-ary operation of $A$. A term $t$ is {\it idempotent} if $t(x,\ldots,x)\approx x$ holds. An algebra $A$ is {\it idempotent} if every term on $A$ is idempotent.

Let $\alpha$ be an equivalence relation on a set $A$. We denote the class of $x$ with respect to $\alpha$ by $[x]_\alpha$ (or simply by $[x]$ in some cases). An equivalence relation $\alpha$ on $A$ is a {\it congruence} if it is compatible with the algebraic structure. Namely, if $f$ is an $n$-ary basic operation then
$$f(x_1,x_2,\ldots, x_n)\,\alpha \, f(y_1,y_2,\ldots, y_n)$$
provided $x_i\, \alpha\, y_i$ for $i=1,\ldots, n$. 
The set of congruences of $Q$ is a lattice denoted by $Con(A)$ with top element $1_A=A\times A$ and bottom element $0_A=\setof{(x,x)}{x\in A}$. 
Congruences and morphisms are essentially the same thing. Indeed if $h:A\longrightarrow A'$ is a morphism then $\ker{h}=\setof{(x,y)\in A}{h(x)=h(y)}$ is a congruence of $A$. On the other hand if $\alpha$ is a congruence, the factor set is endowed with a natural algebraic structure: indeed for every basic $n$-ary operation $f$ on $A$, the map given by
$$f([x_1]_\alpha,\ldots,[x_n]_\alpha)=[f(x_1,\ldots,x_n)]_\alpha$$
for every $x_1,\ldots, x_n\in A$ is a well-defined $n$-ary operation on $A/\alpha$. The canonical map 
$$A\longrightarrow A/\alpha,\quad x\mapsto [x]_\alpha$$
is a morphism with respect to such structure. 

The congruences of $A/\alpha$ are $\setof{\beta/\alpha}{\alpha\leq\beta\in Con(A)$} where $\beta/\alpha$ is defined by setting
$$[x]_\alpha\, \beta/\alpha\, [y]_\alpha\, \text{ if and only if } x\,\beta\, y$$
for every $x,y\in A$. Note that if $A$ is idempotent then the blocks of congruences of $A$ are subalgebras. 


A variety is a class of algebraic structures closed under subalgebras, homomorphic images and direct products. We denote by $\mathcal{S}(A)$ the set of subalgebras of $A$.

 A {\it Malt'sev term} is a ternary (idempotent) term $m$ satisfying the identities
$$m(x,y,y)\approx m(y,y,x)\approx x.$$
We say that an algebraic structure $A$ has a Malt'sev term if the variety generated by $A$ has a Malt'sev term.

Let $\mathcal{K}$ be a class of idempotent algebraic structures. We say that $\mathcal{K}$ is {\it closed under extensions} if $A\in \mathcal{K}$ provided that $A/\alpha\in \mathcal{K}$ and $[x]_\alpha\in \mathcal{K}$ for every $x\in A$ for some $\alpha\in Con(A)$.

Let us recall some group theoretical terminology. Let $G$ be a group acting on a set $Q$ and $x\in Q$. We denote the {\it pointwise stabilizer} of $x$ by $G_x$ and the orbit of $x$ under the action of $G$ by $x^G$. The group $G$ is {\it semiregular} if the pointwise stabilizers of the action of $G$ are trivial and {\it transitive} if $Q=x^G$ for every $x\in Q$. If $G$ is semiregular and transitive on $Q$, we say that $G$ is {\it regular} on $Q$.

\section{Preliminary results}\label{preliminary}

\subsection{Left quasigroups}

A left quasigroup is a binary algebraic structure $(Q,\cdot,\backslash)$  that satisfies the following identities
$$x\cdot(x\backslash y)\approx y\approx x\backslash (x\cdot y).$$
We define the {\it left and right multiplication mapping} as
$$L_x:y\mapsto x\cdot y,\quad R_x:y\mapsto y\cdot x$$
for every $x\in Q$. The map $L_x$ is a permutation for every $x\in Q$ (note that $x\backslash y=L_x^{-1}(y)$ for every $x,y\in Q$) and so we can introduce the {\it left multiplication group} as $\lmlt(Q)=\langle L_x,\, x\in Q\rangle$. In the following we will denote the $\cdot$ operation just by juxtaposition. We define the {\it set of idempotent elements of $Q$} as $E(Q)=\setof{x\in Q}{xx=x}$. Note that if $[x]_\alpha\in E(Q/\alpha)$ then the block of $x$ with respect to $\alpha$ is a subalgebra of $Q$. We say that $Q$ is:

%

\begin{itemize}
\item[(i)] {\it idempotent} if $Q=E(Q)$, i.e. the identity $xx\approx x$ holds.
\item[(ii)] {\it Projection} if the identity $xy\approx y$ holds. We denote by $\mathcal{P}_n$ the projection left quasigroups with $n$ elements.
\item[(iii)] A {\it rack} if the identity $x(yz)\approx (xy)(xz)$ holds (or equivalently $L_x\in \aut{Q}$ for every $x\in Q$). Idempotent racks are called {\it quandles}.
\item[(iv)] {\it Latin} if the right multiplications are bijective. In this case a binary operation can be defined as $x/y=R_y^{-1}(x)$ for $x,y\in Q$. If $Q$ is infinite, the universal algebraic features (e.g. congruences, subalgebras etc) of the associated {\it quasigroup} structure $(Q,\cdot,\backslash, /)$ might be different from the one of the left quasigroup $(Q,\cdot,\backslash)$. 
\end{itemize}

\subsection{Congruences and subgroups}

Let $Q$ be a left quasigroup and $\alpha$ be a congruence of $Q$. Then we have a canonical surjective group homomorphism, defined on generators as
\begin{align}\label{pialpha}
\pi_\alpha:\lmlt(Q)\longrightarrow \lmlt(Q/\alpha),\quad L_x\mapsto L_{[x]_\alpha}.
\end{align}
In particular, we have that $\pi_\alpha(h)([x]_\alpha)=[h(x)]_\alpha$ for every $x\in Q$ and every $h\in\lmlt(Q)$. The kernel of the map \eqref{pialpha} is denoted by $\lmlt^\alpha$. We also define the {\it displacement group relative to $\alpha$} as
\begin{align}\label{disalpha}
\dis_{\alpha}=\langle hL_x L_y^{-1} h^{-1}, \, x\,\alpha\, y,\, h\in \lmlt(Q)\rangle,
\end{align}
namely the normal closure of the set $\setof{L_x L_y^{-1}}{x\, \alpha\, y}$ in $\lmlt(Q)$. In particular we denote by $\dis(Q)$ the displacement group relative to $1_Q$ and by $\dis^\alpha=\dis(Q)\cap\lmlt^\alpha$. We also define the {\it blockwise stabilizer of $x\in Q$} as $\dis(Q)_{[x]_\alpha}=\pi_{\alpha}^{-1}(\dis(Q/\alpha)_{[x]})=\setof{h\in \dis(Q)}{h(x)\,\alpha\, x}$. In particular note that $\dis(Q)_x  \leq \dis(Q)_{[x]_\alpha}$ for every $x\in Q$ and
\begin{align}\label{eq for dis}
\dis_\alpha\leq \dis^\alpha&=\bigcap_{[x]\in Q/\alpha} \dis(Q)_{[x]_\alpha}=\setof{h\in \dis(Q)}{h(x)\, \alpha\, x \text{ for every } x\in Q}. 
\end{align}

The subgroups $\dis_\alpha$ and $\dis^\alpha$ can be defined in the very same way as in \eqref{disalpha} and as in \eqref{eq for dis} for every binary relation $\alpha$ on $Q$. Therefore we have two operators $\dis_*$ and $\dis^*$ from the equivalence relations on $Q$ to the subgroups of $\dis(Q)$ (we can also define the operator $\lmlt^*:\alpha\longrightarrow \lmlt^\alpha\subseteq \lmlt(Q)$).

\begin{lemma}\label{pi}	\cite[Proposition 3.2]{CP}
Let $Q$ be a left quasigroup and $\alpha,\beta\in Con(Q)$ such that $\alpha\leq \beta$. Then: 
$$\dis_{\beta/\alpha}=\pi_{\alpha}(\dis_\beta),\quad \dis^{\beta/\alpha}=\pi_{\alpha}(\dis^\beta),\quad \lmlt^{\beta/\alpha}=\pi_{\alpha}(\lmlt^\beta).$$
\end{lemma}

In particular, the map \eqref{pialpha} restricts and corestricts to $\dis(Q)$ and $\dis(Q/\alpha)$ and the kernel of the restricted map is $\dis^\alpha$.

\begin{proposition}\label{p:dis_alpha1}
Let $Q$ be a left quasigroup and $\setof{\alpha_i}{i\in I}\subseteq Con(Q)$, $\beta=\bigwedge_{i\in I} \alpha_i$ and $\gamma=\bigvee_{i\in I} \alpha_i$. Then:
\begin{enumerate}
\item[(i)] $\dis(Q)_{[x]_\beta} = \bigcap_{i\in I} \dis(Q)_{[x]_{\alpha_i}}$ for every $x\in Q$.
	\item[(ii)] $\dis^{\beta} = \bigcap_{i\in I} \dis^{\alpha_i}$. 
	\item[(iii)]  $\dis_{\gamma}  =  \langle \dis_{\alpha_i},\, i\in I\rangle$. 
\end{enumerate}
\end{proposition}

\begin{proof}

(i) We have that
\begin{align}\label{eq for blockstab}
\dis(Q)_{[x]_\beta}&=\setof{h\in \dis(Q)}{h(x)\, \beta\, x}\\
&=\setof{h\in \dis(Q)}{h(x)\, \alpha_i\, x \text{ for all } i\in I}=\bigcap_{i\in I} \setof{h\in \dis(Q)}{h(x)\, \alpha_i\, x  }\notag \\
&=\bigcap_{i\in I} \dis(Q)_{[x]_{\alpha_i}}.\notag
\end{align}

(ii) According to \eqref{eq for dis} and \eqref{eq for blockstab}, we have $\dis^\beta=\bigcap_{i\in I} \dis^{\alpha_i}$.

(ii) It is easy to see that $\dis_{\alpha_i}\leq \dis_{\gamma}$, and thus $\langle \dis_{\alpha_i},\, i\in I\rangle\leq \dis_{\gamma}$.  
For the other inclusion, let $x\,\gamma\,y$, and take the witnesses $x=z_1,\dots,z_n=y$ such that $z_k\,\alpha_{j_k}\,z_{k+1}$, for some $j_k\in I$ and every every $k$. Then
\[ L_x L_y^{-1}=L_{z_1} L_{z_n}^{-1}=\underbrace{L_{z_1}L_{z_2}^{-1}}_{\in\dis_{\alpha_{j_1}}}\underbrace{L_{z_2}L_{z_3}^{-1}}_{\in\dis_{\alpha_{j_2}}}\underbrace{L_{z_3}L_{z_4}^{-1}}_{\in\dis_{\alpha_{j_3}}}\cdots\underbrace{L_{z_{n-1}}L_{z_n}^{-1}}_{\in\dis_{\alpha_{j_{n-1}}}}\]
and thus every generator $fL_x L_y^{-1}f^{-1}$ of $\dis_{\gamma}$ belongs to $\langle \dis_{\alpha_i},\, i\in I\rangle$.
\end{proof}


For a left quasigroup $Q$ and a subgroup $N\leq \lmlt(Q)$ we can define two equivalence relations as
\begin{align*} 
 x\, \mathcal{O}\, y\, & \text{if and only if } x=h(y)\, \text{ for some } h\in N,\\
 x\, \c{N}\, y &\text{ if and only if } L_x L_y^{-1} \in N.
\end{align*}
Hence we have two operators $\c{*}$ and $\mathcal{O}_*$ from the set of subgroups of $\lmlt(Q)$ to the set of the equivalence relations on $Q$.

We say that $Q$ is {\it connected by $N$} if $N$ is transitive on $Q$, i.e. $\mathcal{O}_N=1_Q$. If $\lmlt(Q)$ is transitive we simply say that $Q$ is {\it connected}. Note that if $Q$ is a connected idempotent left quasigroups or a left quasigroup with a Mal'cev term then $Q$ is connected by $\dis(Q)$ \cite[Proposition 3.6]{Maltsev_paper}. If all the subalgebras of $Q$ are connected we say that $Q$ is {\it superconnected} (in particular $Q$ is connected). The class of superconnected left quasigroups is closed under subalgebras and homomorphic images and the class of superconnected idempotent left quasigroups is also closed under extensions \cite[Corollary 1.12]{Super}. In particular, finite latin left quasigroups are superconnected but the converse is not true in general (see \cite[Example 1.8(ii)]{Super} for an example of an infinite latin quandle that is not superconnected).

The {\it admissible subgroups} of $Q$, as defined in \cite{LTT}, are $$\N(Q)=\setof{N\trianglelefteq \lmlt(Q)}{\mathcal{O}_N\leq \c{N}}=\setof{N\trianglelefteq \lmlt(Q)}{\dis_{\mathcal{O}_N}\leq N}.$$
The admissible subgroups form a sublattice of the lattice of the normal subgroups of $\lmlt(Q)$. If $N\in \N(Q)$ then $\mathcal{O}_N$ is a congruence of $Q$ and $\dis_\alpha,\dis^\alpha,\lmlt^\alpha \in \N(Q)$ \cite[Lemma 1.7, Corollary 1.9]{LTT}. It is easy to verify that 
$$\mathcal{O}_{\dis_\alpha}\leq \mathcal{O}_{\dis^\alpha}\leq \mathcal{O}_{\lmlt^\alpha} \leq \alpha\leq \c{\dis_\alpha}\leq \c{\dis^\alpha}$$
for every congruence $\alpha$. 

%
%
%
%
%
%

%

%

\begin{lemma}\label{blocks connected}
Let $Q$ be an idempotent left quasigroup and $\alpha\in Con(Q)$. If the blocks of $\alpha$ are connected then $\alpha=\mathcal{O}_{\dis_\alpha}=\mathcal{O}_{\dis^\alpha}$.
\end{lemma}

\begin{proof}
In general we have $\mathcal{O}_{\dis_\alpha}\leq \mathcal{O}_{\dis^\alpha}\leq \alpha$. Let $x\in Q$ and consider the groups $H=\langle L_y,\, y\in [x]\rangle$ and $D=\langle h L_y L_z^{-1} h^{-1}, \, y,z\in [x],\, h\in H\rangle\leq \dis_\alpha$. The action of $D$ coincides with the action of $\dis([x])$ that is transitive on $[x]$. Therefore, $\alpha=\mathcal{O}_{\dis_\alpha}$.
\end{proof}

The following Proposition shows that the correspondence theorem for normal subgroups restricts to the sublattices of admissible groups.

\begin{proposition}\label{iso of lattices}
Let $Q$ be a left quasigroup and $\alpha$ be a congruence of $Q$. Then the mappings 
\begin{align*}
\setof{N\in \N(Q)}{\lmlt^\alpha\leq N}&\longleftrightarrow \N(Q/\alpha)\\
N&\mapsto \pi_\alpha(N)\\
\pi_\alpha^{-1}(K)&\leftmapsto K
\end{align*}
provides an isomorphism of lattices.
\end{proposition}

\begin{proof}
The pair of maps above clearly define a bijective correspondence between the lattice of normal subgroups of $\lmlt(Q)$ and the lattice of normal subgroups of $\lmlt(Q/\alpha)\cong \lmlt(Q)/\lmlt^\alpha$.

We already proved that if $N\in \N(Q)$ then $\pi_\alpha(N)\in \N(Q/\alpha)$ in \cite[Lemma 1.7]{LTT}. Assume that $K\in \N(Q/\alpha)$ and let $H=\pi_\alpha^{-1}(K)$. Clearly $H$ is normal in $\lmlt(Q)$. If $h\in H$ we have
$$\pi_\alpha (L_{h(x)} L_x^{-1})=L_{[h(x)]}L_{[x]}^{-1}=L_{\pi_\alpha(h)([x])}L_{[x]}^{-1}\in \dis_{\mathcal{O}_K}\leq K.$$
Thus, $L_{h(x)} L_x^{-1}\in H$ and so $\dis_{\mathcal{O}_H}\leq H$.
\end{proof}

Given a left quasigroup $Q$, note that the images of the operators $\dis^*$ and $\dis_*$ lie in the sublattice 
\begin{align}\label{NN}
\N'(Q)=\setof{N\in \N(Q)}{N\leq \dis(Q)}=\setof{N\cap \dis(Q) }{N\in \N(Q)}.
\end{align}
The correspondence given in Proposition \ref{iso of lattices} can be restricted to this sublattice.
\begin{corollary}\label{iso of lattice 2}
Let $Q$ be a left quasigroup and $\alpha$ be a congruence of $Q$. Then the mappings 
\begin{align*}
\setof{N\in \N'(Q)}{\dis^\alpha\leq N} &\longleftrightarrow \N'(Q/\alpha)\\
N&\mapsto \pi_\alpha(N)\\
\pi_\alpha^{-1}(K)&\leftmapsto K
\end{align*}
provides an isomorphism of lattices.
\end{corollary}

The lattice of congruences of a left quasigroups and the lattice of admissible subgroups are related by a monotone Galois connection. Since the image of the operator $\dis^*$ lie in the sublattice defined in \eqref{NN} we can restate \cite[Theorem 1.10]{LTT} as follows.

\begin{theorem}\label{galois_connection}\cite[Theorem 1.10]{LTT}
Let $Q$ be a left quasigroup. The pair of mappings $\mathcal{O}_*$ and $\dis^{*}$ provides a monotone Galois connection between $Con(Q)$ and $\N'(Q)$.
\end{theorem}

%


\subsection{The Cayley kernel}
%

Let $Q$ be a left quasigroup, we introduce the {\it Cayley kernel of $Q$} as the equivalence relation $\lambda_Q=\c{1}$. Namely, the Cayley kernel is defined by setting
$$x\, \lambda_Q\, y\, \text{ if and only if } L_x=L_y.$$
The Cayley kernel is not a congruence in general. If $\lambda_Q$ is a congruence we say that $Q$ is a {\it Cayley left quasigroup} (e.g. racks are Cayley left quasigroups).

\begin{remark}\label{remark on lambda}
Let us point out two easy properties of the Cayley kernel of a left quasigroup $Q$. Let $\alpha\in Con(Q)$:
\begin{itemize}
\item[(i)] $\alpha\leq \lambda_Q$ if and only if $\dis_\alpha=1$. 
\item[(ii)] If $x\,\lambda_Q\, y$ then $[x]_\alpha\,\lambda_{Q/\alpha}\, [y]_\alpha$. Indeed: if $L_x=L_y$, then $L_{[x]_\alpha}=\pi_\alpha(L_x)=\pi_\alpha(L_y)= L_{[y]_\alpha}$.
\item[(iii)] According to  \cite[Proposition 1.6]{LTT}, $\lambda_{Q/\alpha}=\c{\dis^\alpha}/\alpha$.
\end{itemize}
\end{remark}

The Cayley kernel plays a role in the universal algebraic theory of left quasigroups. Indeed in \cite{covering_paper} we established that {\it strongly abelian congruences} of left quasigroups in the sense of \cite{TCT} are those below the Cayley kernel. Such congruences form a complete sublattice of the lattice of congruences. 

\begin{proposition}
Let $Q$ be a left quasigroup. The set $\setof{\alpha\in Con(Q)}{\alpha\leq \lambda_Q}$ is a complete sublattice of $Con(Q)$.
\end{proposition}


\begin{proof}
Let $\setof{\alpha_i}{i\in I}\subseteq Con(Q)$ and assume that $\alpha_i\leq \lambda_Q$ for every $i\in I$, namely $\dis_{\alpha_i}=1$ for every $i\in I$. Clearly $\wedge_{i \in I}\alpha_i\leq \lambda_Q$. Let $\beta=\vee_{i\in I} \alpha_i$.  According to Proposition \ref{p:dis_alpha1}(ii) $\dis_{\beta}=1$, namely $\beta\leq \lambda_Q$ by Remark \ref{remark on lambda}(i). 
\end{proof}

The congruences defined by orbits are related to the Cayley kernel.
%
%
%
%
%

\begin{lemma}\label{lemma below lambda}
Let $Q$ be a left quasigroup, $\alpha\in Con(Q)$, $\beta=\mathcal{O}_{\dis^\alpha}$ and $\gamma=\mathcal{O}_{\dis_\alpha}$. Then:
\begin{itemize}

\item[(i)] $\dis^\alpha=\dis^\beta$ and $\alpha/\beta\leq \lambda_{Q/\beta}$

\item[(ii)] $\alpha/\gamma\leq \lambda_{Q/\gamma}$.
\end{itemize}
\end{lemma}

\begin{proof}

(i) The pair $\mathcal{O}_*$ and $\dis^*$ provides a Galois connection, therefore $\dis^\beta=\dis^{\mathcal{O}_{\dis^\alpha}}=\dis^\alpha$. Moreover, $\dis_\alpha\leq \dis^\alpha=\dis^\beta$ and so $\dis_{\alpha/\beta}=\pi_\beta(\dis_\alpha)=1$. Thus, $\alpha/\beta\leq \lambda_{Q/\beta}$ by Remark \ref{remark on lambda}(i).

(ii) Clearly $\gamma\leq \alpha$ and $\dis_\alpha\leq \dis^\gamma$. Therefore $\dis_{\alpha/\gamma}=\pi_{\gamma}(\dis_\alpha)=1$, i.e. $\alpha/ \gamma\leq \lambda_{Q/\gamma}$ (see Remark \ref{remark on lambda}(i)).
\end{proof}

%

%

Let $Q$ be a left quasigroup. If $\lambda_Q=0_Q$ we say that $Q$ is {\it faithful}. The class of faithful left quasigroups is closed under direct products and the class of idempotent faithful left quasigroups is closed under extensions \cite[Corollary 1.12]{Super}.

\begin{lemma}\label{factor of lambda}
Let $Q$ be a left quasigroup, $\alpha\in Con(Q)$, $\setof{\alpha_i}{i\in I}\subseteq Con(Q)$ and $\beta=\wedge_{i\in I}\alpha_i$. 
\begin{itemize}
\item[(i)] If $Q/\alpha$ is faithful then $\lambda_Q\leq \alpha$.
\item[(ii)] If $Q/\alpha_i$ is faithful for every $i\in I$ then $Q/\beta$ is faithful.
\end{itemize}

\end{lemma}

\begin{proof}
%
%

(i) Let $L_x=L_y$.  Then by Remark \ref{remark on lambda}(ii) we have $L_{[x]_\alpha}=L_{[y]_\alpha}$ and so $[x]_\alpha=[y]_\alpha$ since $Q/\alpha$ is faithful.

(ii) According to Remark \ref{remark on lambda}(iii), $\alpha_i=\c{\dis^{\alpha_i}}$ for every $i\in I$ and we need to prove that $\beta=\c{\dis^\beta}$. Using that $\dis^\beta=\bigcap_{i\in I} \dis^{\alpha_i}$, we have $\c{\dis^\beta}=\wedge_{i\in I} \c{\dis^{\alpha_i}}=\wedge_{i\in I} \alpha_i=\beta$.
%
%
\end{proof}

The converse of Lemma \ref{factor of lambda}(i) does not hold. Given a left quasigroup $Q$ and $\alpha=\mathcal{O}_{\dis(Q)}$ then $\lambda_{Q/\alpha}=1_Q$ \cite[Corollary 1.9]{LTT}. It is easy to construct a left quasigroup such that $\lambda_Q\leq \alpha$ and that $Q/\alpha$ is not trivial (e.g. a faithful left quasigroup that is not connected by its displacement group).

Let $Q$ be a left quasigroup. If all the subalgebras of $Q$ are faithful, $Q$ is said to be {\it superfaithful}. The class of superfaithful left quasigroups is closed under subalgebras and the class of idempotent superfaithful left quasigroups is closed under extensions \cite[Corollary 1.12]{Super}.

\begin{lemma}	\label{lemma on super}
Let $Q$ be a left quasigroup. 
\begin{itemize}
\item[(i)] If $Q$ has injective right multiplications then $Q$ is superfaithful.
\item[(ii)] If $Q$ is superfaithful then $\mathcal{P}_2\notin \mathcal{S}(Q)$. 
\item[(iii)] If $Q$ is idempotent and $\mathcal{P}_2\notin \mathcal{S}(Q)$ then $Q$ is superfaithful.
\end{itemize}
\end{lemma}

\begin{proof}
(i) The statement follows by the same proof of \cite[Lemma 1.1]{LTT}.

(ii) The left quasigroup $\mathcal{P}_2$ is not faithful.

(iii) It follows by \cite[Lemma 1.9]{Super}.
\end{proof}

%

According to Lemma \ref{lemma on super} latin left quasigroups are superfaithful and the class of superfaithful idempotent left quasigroups is the class of idempotent left quasigroups such that the following implication holds:
\begin{equation}\label{super implication}
xy=y \, \text{ and } \, yx=x\,\Rightarrow \, x=y.
\end{equation}

Let $Q$ be an idempotent left quasigroup $Q$. Then $Fix(L_x)=\{x\}$ for every $x\in Q$ if and only if the implication
\begin{equation}\label{Fix}
xy=y \,\Rightarrow \, x=y
\end{equation}
holds. The class of idempotent left quasigroups satisying \eqref{super implication} (resp. \eqref{Fix}) is closed under subalgebras and direct products as it is a quasi-variety \cite{Burris} and extensions. Indeed assume that $Q/\alpha$ and $[x]$ are is such class and $xy=y$. Then $[x][y]=[y]$. So $[x]=[y]$ and thus $x=y$. 

\begin{corollary}\label{P_2 and idempotent}
Let $Q$ be an idempotent left quasigroup. 
\begin{itemize}
\item[(i)] If right multiplications are injective, then $Fix(L_x)=\{x\}$ for every $x\in Q$.

\item[(ii)] If $Fix(L_x)=\{x\}$ for every $x\in Q$ then $Q$ is superfaithful.
\end{itemize}

\begin{proof}

(i) Assume that $xy=y$. Then $R_y(x)=xy=y=yy=R_y(y)$ and so $x=y$.

(ii) Clear since \eqref{Fix} implies \eqref{super implication}.
\end{proof}

\end{corollary}
%

According to Lemma \ref{lemma on super}(iii), superconnected idempotent left quasigroups are superfaithful. In particular we have the following.

\begin{lemma}\label{O onto super}
Let $Q$ be a superconnected idempotent left quasigroup. Then $\mathcal{O}_{\dis_\alpha}=\mathcal{O}_{\dis^\alpha}=\alpha=\c{\dis_\alpha}=\c{\dis^\alpha}$ for every $\alpha\in Con(Q)$.
\end{lemma}

\begin{proof}
Let $\alpha\in Con(Q)$. The blocks of $\alpha$ are connected, so according to Lemma \ref{blocks connected} we have $\alpha=\mathcal{O}_{\dis_\alpha}=\mathcal{O}_{\dis^\alpha}$. According to Remark \ref{remark on lambda}(iii) we have that $\lambda_{Q/\alpha}=\c{\dis^\alpha}/\alpha$. The left quasigroup $Q$ and its factors are faithful, and so $\alpha=\c{\dis^\alpha}$.
%
%
\end{proof}

\section{Semiregular left quasigroups}\label{Semiregular section}

\subsection{Semiregular left quasigroups}

 We say that a left quasigroup $Q$ is {\it semiregular} if $\dis(Q)$ is semiregular on $Q$. A class of semiregular left quasigroups is the class of {\it principal quandles} studied in \cite{Principal}. Some of the results that hold for semiregular quandles can be easily extended to left quasigroups.

 It is easy to see that the class of semiregular left quasigroups is closed under subalgebras and direct products but it is not closed under homomorphic images. For instance, the quandle {\tt SmallQuandle(12,1)} in the RIG database of GAP is semiregular, but it has a factor that is not semiregular \cite{GAP4}. The next lemma characterizes semiregular factors of a left quasigroup.
 
\begin{lemma}\label{semiregular factor}
Let $Q$ be a left quasigroup and $\alpha$ be a congruence of $Q$. 
\begin{itemize}
\item[(i)] $Q/\alpha$ is semiregular if and only if $\dis^\alpha=\dis(Q)_{[x]}$ for every $x\in Q$.
\item[(ii)]  Let $\beta= \bigwedge_{i\in I} \alpha_i\in Con(Q)$. If $Q/\alpha_i$ is semiregular for every $i\in I$, then $Q/\beta$ is semiregular.
\end{itemize}
\end{lemma}

\begin{proof}
(i) Since $\dis^\alpha\leq \dis(Q)_{[x]}=\pi_\alpha^{-1}(\dis(Q/\alpha)_{[x]})$, we have that $\dis(Q/\alpha)_{[x]}=1$ if and only if $\dis(Q)_{[x]}= \dis^\alpha$. 
%
%
%
%
%
%

(ii) Let $\beta=\bigwedge_{i\in I}\alpha_i$. Then $Q/\beta$ embeds into the direct product of semiregular left quasigroups $\prod_{i\in I} Q/\alpha_i$ and so it is semiregular as well.
%
\end{proof}

For semiregular left quasigroups, the property of being faithful and superfaithful collapse.

\begin{lemma}\label{lemma1}
Let $Q$ be a semiregular left quasigroup. The following are equivalent: 
\begin{itemize}
\item[(i)] $Q$ is faithful.
\item[(ii)] The right multiplication mappings of $Q$ are injective. 
\item[(iii)] $Q$ is superfaithful.
\end{itemize}

\end{lemma}

\begin{proof}
The implication (iii) $\Rightarrow$ (i) is clear and (ii) $\Rightarrow$ (iii) is Lemma \ref{lemma on super}(i).

(i) $\Rightarrow$ (ii) Let $xz=yz$. Then $L_y^{-1} L_x\in \dis(Q)_z=1$ and so $L_x=L_y$. Thus $x=y$.
\end{proof}

\begin{lemma}\label{semiregular idempotent are quandles}
Semiregular idempotent left quasigroups are quandles.
\end{lemma}

\begin{proof}
Let $Q$ be an idempotent semiregular left quasigroup and $x,y\in Q$. According to \cite[Lemma 1.4]{LTT} $L_x L_y L_x^{-1} L_{xy}^{-1}\in \dis(Q)$. Moreover
$$L_x L_y L_x^{-1} L_{xy}^{-1}(xy)=xy$$
and so $L_x L_y L_x^{-1} L_{xy}^{-1}\in \dis(Q)_{xy}=1$. Therefore $L_{xy}=L_x L_y L_x^{-1}$, i.e. $Q$ is a quandle.
\end{proof}
\subsection{The equivalence $\sigma_Q$}

Let $Q$ be a left quasigroup and $N\trianglelefteq \lmlt(Q)$. We define an equivalence relation as
$$x \, \sigma_N\, y \text{ if and only if } N_x=N_y.$$

\begin{lemma}
Let $Q$ be a left quasigroup and $N\trianglelefteq \lmlt(Q)$. The blocks of $\sigma_N$ are blocks with respect to the action of $\lmlt(Q)$. If $Q$ is idempotent they are subalgebras of $Q$.
\end{lemma}

\begin{proof}
Let $x\, \sigma_N \, y$, i.e. $ N_x=N_y$ and $h\in \lmlt(Q)$. The group $N$ is normal in $\lmlt(Q)$, so we have $N_{h(x)}=h N_x h^{-1}=h N_y h^{-1} = N_{h(y)}$. Hence $h(x)\, \sigma_N\, h(y)$ and so the classes of $\sigma_N$ are blocks with respect to the action of $\lmlt(Q)$. If in addition $Q$ is idempotent, then $$N_{L_x^{\pm 1}(y)}=L_x^{\pm 1} N_y L_x^{\mp 1}=L_x^{\pm 1} N_x L_x^{\mp 1}=N_{L_x^{\pm 1}(x)}=N_x,$$ so $[x]_{\sigma_N}$ is a subalgebra of $Q$. 
\end{proof}

 In \cite{CP} we already introduced the equivalence relation $\sigma_{\dis(Q)}$ and we denoted it simply by $\sigma_Q$. 
Semiregularity of $Q$ is captured by the equivalence $\sigma_Q$. Indeed $Q$ is semiregular if and only if $\sigma_Q=1_Q$.

Such a relation is related to central congruences of left quasigroups (see \cite{CP}) and its properties have been investigated for quandles in \cite{Principal}. Some of its properties can be extended to idempotent left quasigroups.

\begin{proposition}\label{semiregular decomposition}
Let $Q$ be an idempotent left quasigroup. The classes of $\sigma_Q$ are semiregular subquandles of $Q$. In particular every idempotent left quasigroup is a disjoint union of semiregular quandles.
\end{proposition}
\begin{proof}

Let $H=\langle L_y,\, y\in [x]_{\sigma_Q}\rangle$. The action of the displacement group of $[x]_{\sigma_Q}$ is the action of the group $D=\langle h L_y L_x^{-1} h^{-1}, \,y\in [x]_{\sigma_Q},\, h\in H\rangle $ restricted to $[x]_{\sigma_Q}$. So if $h\in \dis([x]_{\sigma_Q})$, then $h=g|_{[x]_{\sigma_Q}}$ for 
some $g\in D$. If $h(x)=x$ then $h(y)=g(y)=y$ for every $y\in [x]_{\sigma_Q}$, i.e. $h=1$ and $[x]_{\sigma_Q}$ is semiregular. According to Lemma \ref{semiregular idempotent are quandles} the classes of $\sigma_Q$ are semiregular quandles.
\end{proof}

Let $Q$ be a left quasigroup, $x\in Q$ and let $\widetilde{N_x}=N_{\dis(Q)}(N_x)$. Then $[x]^{\dis(Q)}\cap [x]_{\sigma_N}=x^{\widetilde{N_x}}$. Indeed $h(x)\, \sigma_N\, x$ if and only if $N_{h(x)}=h N_x h^{-1} =N_x$, i.e. $h\in \widetilde{N_x}$. The equivalence $\sigma_N$ can be trivial and in this case $\widetilde{N_x}=N_x$ and $Z(\dis(Q))=1$ ($Z(\dis(Q))\leq \widetilde{N_x}$ for every $x\in Q$).

\begin{corollary}\label{decomposition}
Let $Q$ be a connected idempotent left quasigroup and $N\in \N(Q)$. Then $[x]_{\sigma_N}=x^{\widetilde{N_x}}$ for every $x\in Q$. 
\end{corollary}
\begin{proof}
Since $Q$ is connected, $x^{\widetilde{N_x}}=[x]_{\sigma_N}\cap x^{\dis(Q)}=[x]_{\sigma_N}$. 
\end{proof}

\subsection{Semiregular admissible subgroups}\label{semiregular subgroups}

Let us consider admissible semiregular subgroups. Note that, if $Q$ is a left quasigroup and $N$ is an admissible subgroup, then $N$ is semiregular if and only if $\sigma_N=1_Q$. Indeed, if $N_x=1$ for every $x\in Q$, we have $\sigma_N=1_Q$. On the other hand, if $\sigma_N=1_Q$ and $h\in N_x$ then $h(y)=y$ for every $y\in N$. Thus $h=1$ and so $N$ is semiregular.

\begin{lemma}\label{semiregul groups_0}
Let $Q$ be a left quasigroup, $N\in Norm^\prime(Q)$ and $\alpha=\mathcal{O}_N$. If $N$ is semiregular, then $\dis(Q)_{[x]_{\alpha}} \cong N\rtimes \dis(Q)_x$ and $\dis^{\alpha}\cong N\rtimes \dis^{\alpha}_x$ for every $x\in Q$.
%
%
%
%
%
%
%
%
%
%
%
\end{lemma}

\begin{proof}
Let $x\in Q$. Recall that $\dis_{\alpha}\leq N\leq \dis^\alpha\leq \dis(Q)_{[x]}$ and $\dis(Q)_x\leq \dis(Q)_{[x]}$. The subgroup $N$ is a normal subgroup of $\BlocS{Q}{x}$. The group $N$ is regular on $[x]$, so the elements of $N$ are representatives of the cosets with respect to $\dis(Q)_x$ and so $\BlocS{Q}{x}=N \dis(Q)_x$. Moreover $N\cap \dis(Q)_x=1$, therefore the block stabilizer splits as a semidirect product of the $N$ and $\dis(Q)_x$. The same argument shows that a similar decomposition holds for $\dis^{\alpha}$.
%
%
%
%
%
%
%
%
%
 %
\end{proof}

\begin{lemma}\label{semiregul groups}
Let $Q$ be a finite faithful left quasigroup, $N\in Norm^\prime(Q)$ and $\alpha=\mathcal{O}_N$. If $N$ is semiregular, then:
\begin{itemize}

\item[(i)] $\alpha=\c{N}$ and $N=\dis_{\alpha}=\{L_y L_x^{-1}, y \in [x]\}$ for every $x\in Q$.

\item[(ii)] If $[x]\in  E(Q/\alpha)$, then $[x]$ is a semiregulae latin left quasigroup and $N \cong \dis([x])$.
\end{itemize}

\end{lemma}

\begin{proof}

(i) Let $x\in Q$ and let $\varphi$ be the mapping:
\begin{equation*}
\varphi :N\longrightarrow \dis_{\mathcal{O}_N}\leq N, \quad n\mapsto  L_{n(x)} L_{x}^{-1}.
\end{equation*}
If $L_{n(x)}L_x^{-1}=\varphi(n)=\varphi(m)=L_{m(x)} L_x^{-1}$ then $n(x)=m(x)$ since $Q$ is faithful. The subgroup $N$ is semiregular and so $n=m$. Therefore $\varphi$ is bijective and so $N=\setof{L_{n(x)} L_x^{-1}}{n\in N} \leq \dis_{\alpha}$. Since it is always the case that 
 $ \dis_{\alpha} \leq N$, we can conclude that $N=\dis_{\alpha}$. So, if $y\, \c{N}\, x$ (i.e. $L_y L_x^{-1}\in N$) then $y=n(x)$ for some $n\in N$. Thus $\mathcal{O}_N=\c{N}$.


(ii) If $[x]\in E(Q/\alpha)$, the block $[x]$ is a subalgebra. Let $H=\langle L_y,\, y\in [x]\rangle$. The subgroup $N$ is normal in $\lmlt(Q)$ and so by item (i) $N=\langle h L_y L_x^{-1} h^{-1},\, x\in [y],\, h\in H\rangle=\setof{L_y L_x^{-1}}{y\in [x]}$. So, we have 
$$\dis([x])\cong \langle h L_y L_x^{-1}h^{-1}|_{[x]}, h\in H, y\in [x]\rangle|_{[x]}=N|_{[x]}.$$

 Since the group $N$ is semiregular, then $K=\bigcap_{y\in [x]} N_y=1$ and so $\dis([x])\cong N|_{[x]}\cong N/K=N$. In particular $[x]$ is semiregular.

Assume that $yx=zx$ for $y,z\in [x]$. Then $L_y^{-1} L_z\in (\dis_\alpha)_x=1$ and using that $Q$ is faithful we have $y=z$. Therefore the subalgebra $[x]$ is latin. 
\end{proof}

\begin{corollary}\label{semiregul groups_2}
Let $Q$ be a finite faithful idempotent left quasigroup and let $N\in Norm^\prime (Q)$ be semiregular. Then:
\begin{itemize}
\item[(i)] the blocks of $\mathcal{O}_N$ are semiregular latin subquandles of $Q$ and $N$ is solvable.
\item[(ii)] If $Q/\alpha$ is superfaithful (resp. superconnected, resp. $Fix(L_{[x]})=\{[x]\}$ for every $[x]\in Q/\alpha$) then $Q$ is superfaithful (resp. superconnected, resp. $Fix(L_{[x]})=\{x\}$ for every $x\in Q$). 
\end{itemize}
\end{corollary}


\begin{proof}
(i) According to Proposition \ref{semiregul groups}, the blocks are latin semiregular subalgebras and so they are quandles by Lemma \ref{semiregular idempotent are quandles}. The subgroup $N$ is isomorphic to the displacement group of the blocks, that is solvable \cite[Theorem 1.3]{Stein2}.
%
%
%
%
%

(ii) Such properties are closed under extensions and the blocks of $\alpha$ are finite latin quandles (see item (i)), so they have such properties.
\end{proof}

%
%
%
%
%
%
%
%
%
%
%
%
%
%
%

\section{Nilpotent left quasigroups}\label{nilpotent section}

\subsection{Commutator theory}

The notion of commutator of congruences and the related concepts of center, solvability and nilpotency have been developed for arbitrary algebraic structure in \cite{comm}. 

Let $A$ be an algebraic structure and $\alpha,\beta,\delta\in Con(A)$. We say that \emph{$\alpha$ centralizes $\beta$ over $\delta$}, and write $C(\alpha,\beta;\delta)$, if for every $(n+1)$-ary term operation $t$, every pair $x\,\alpha\,y$ and every $z_1\,\beta\,u_1$, $\dots$, $z_n\,\beta\,u_n$ we have
\[  t(x,z_1,\dots,z_n)\,\delta \, t(x,u_1,\dots,u_n)\quad\text{implies}\quad t(y,z_1,\dots,z_n) \, \delta \, t(y,u_1,\dots,u_n). \]

The \emph{commutator} of $\alpha,\beta\in Con(A)$, denoted by $[\alpha,\beta]$, is the smallest congruence $\delta$ such that $C(\alpha,\beta;\delta)$. 
A congruence $\alpha$ is called:
\begin{itemize}
	\item \emph{abelian} if $C(\alpha,\alpha;0_A)$, i.e., if $[\alpha,\alpha]=0_A$,
	\item \emph{central} if $C(\alpha,1_A;0_A)$, i.e., if $[\alpha,1_A]=0_A$.
\end{itemize}
The \emph{center} of $A$, denoted by $\zeta_A$, is the largest congruence of $A$ such that $C(\zeta_A,1_A;0_A)$. An algebraic structure $A$ is called \emph{abelian} if $\zeta_A=1_A$, or, equivalently, if the congruence $1_A$ is abelian. We can define a series of congruence of $A$ as
$$\zeta_1(A)=\zeta_A,\qquad \zeta_{n+1}(A)/\zeta_{n}(A)=\zeta_{(A/\zeta_{n}(A))}$$
for every $n\in \mathbb{N}$. The algebraic structure $A$ is called \emph{nilpotent} of length $n$ if $\zeta_n(A)=1_A$.
%
%

In \cite{CP} we adapted the universal algebraic theory of commutators to the setting of racks and quandles. The main results of our work can be partially applied to the setting of left quasigroups.

\begin{lemma}\cite[Lemma 5.1]{CP}\label{from CP}
Let $Q$ be a left quasigroup, $\alpha,\beta$ its congruences. If $C(\alpha,\beta;0_Q)$ holds then $[\dis_{\alpha},\dis_{\beta}]=1$ and $\alpha\leq \sigma_{\dis_\beta}$.
\end{lemma}

\begin{corollary}\label{corollary from CP}
Let $Q$ be a left quasigroup and $\alpha\leq \zeta_Q$. Then $\dis_\alpha$ is central in $\dis(Q)$ and $\alpha\leq \sigma_Q$. 
\end{corollary}

According to Lemma \ref{from CP}, abelian left quasigroups are semiregular, indeed we have the following Corollary applying directly Lemma \ref{lemma1} and Lemma \ref{semiregular idempotent are quandles}.

\begin{corollary}\label{corollary for abelian}
Let $Q$ be a abelian left quasigroup. 
\begin{itemize}
\item[(i)] If $Q$ is idempotent then $Q$ is a quandle.
\item[(ii)] If $Q$ is finite and faithful then $Q$ is latin.
\end{itemize}
\end{corollary}

\subsection{Central Congruences}

In this paper we exploit the results of the Section \ref{semiregular subgroups} to study central congruences. The key observation is the following: let $Q$ be a left quasigroup connected by $\dis(Q)$ and $N\leq Z(\dis(Q))$, then $N$ is semiregular. And so, in particular if $\alpha$ is a central congruence then $\dis_\alpha$ is semiregular (see Corollary \ref{corollary from CP}). For this reason the results of this section are stated for left quasigroups connected by their displacement group (and so they cover connected idempotent left quasigroup and left quasigroups with a Malt'sev term).

Central congruences are reflected by the structure of the displacement group.

\begin{lemma}\label{central orbits}
Let $Q$ be a left quasigroup connected by $\dis(Q)$, $\alpha$ be a central congruence and $\beta=\mathcal{O}_{\dis_\alpha}$. Then: 
\begin{itemize}
\item[(i)] $\dis(Q)_{[x]_{\beta}} \cong \dis_\alpha\times \dis(Q)_x$ and $\dis^\beta\cong \dis_\alpha\times \dis^\beta_{x}$ for every $x\in Q$. 
\item[(ii)] $\dis^\beta$ embeds into $\dis_\alpha^{Q/\alpha}$. In particular, $\dis^\beta$ is abelian.
\item[(iii)] If $Q$ is finite, $|\dis^\beta|$ divides $|[x]_\beta|^{|Q/\beta|}$ for every $x\in Q$.
\end{itemize}

\end{lemma}
%
%


\begin{proof}
(i) The congruence $\alpha$ is central and so the group $\dis_\alpha$ is central in $\dis(Q)$ by virtue of Corollary \ref{corollary from CP}. A central subgroup of a transitive group is semiregular, and then $\dis_\alpha$ is regular over the blocks of $\beta$ . According to Lemma \ref{semiregul groups_0} we have that $\dis(Q)_{[x]_\beta}\cong \dis_\alpha\rtimes \dis(Q)_x$ and $\dis^\beta=\dis_\alpha \rtimes \dis^\beta_x$ for every $x\in Q$. The subgroup $\dis_\alpha$ is central, therefore the semidirect product is actually a direct product.
%

(ii) The mapping
\begin{equation}\label{map_2}
\dis^\beta\longrightarrow \prod_{[x]\in Q/\beta} \sym_{[x]}, \quad h\mapsto \setof{h|_{[x]}}{[x]\in Q/\beta}
\end{equation}
is an injective group homomorphism. 
The congruence $\beta$ is central, so $\dis^\beta_x=\dis^\beta_y$ whenever $x\,\beta\, y$. Then the action of $\dis^\beta_x$ on the block of $x$ is trivial. Therefore $\dis^\beta|_{[x]_\beta}=\dis_\alpha|_{[x]_\beta}\cong \dis_\alpha$ 
and so the image of \eqref{map_2} is contained in $\dis_\alpha^{|Q/\beta|}$ which is abelian. 

(iii) If $Q$ is finite, $|[x]_\beta|=|\dis_\alpha|$ and so $|\dis^\beta|$ divides $|[x]|^{|Q/\beta|}$.
\end{proof}

\begin{corollary}\label{strucure_of_K_N}
Let $Q$ be a finite faithful left quasigroup connected by $\dis(Q)$ and $\alpha$ be a central congruence. Then: 
\begin{itemize}
\item[(i)] $\mathcal{O}_{\dis_\alpha}=\alpha=\c{\dis_\alpha}$.
\item[(ii)] $\dis(Q)_{[x]}\cong \dis_\alpha\times \dis(Q)_x$ and $\dis^\alpha \cong \dis_\alpha\times \dis^\alpha_x$ for every $x\in Q$.

\item[(iii)] $\dis^\alpha$ embeds into $\dis([x])^{Q/\alpha}$. In particular, $\dis^\alpha$ is abelian and $|\dis^\alpha|$ divides $|[x]|^{|Q/\alpha|}$ for every $x\in Q$.
\end{itemize}
\end{corollary}

\begin{proof}
Let $x\in Q$. The subgroup $\dis_\alpha$ is a central subgroup of the transitive group $\dis(Q)$, see Corollary \ref{corollary from CP}. Therefore $\dis_{\alpha}$ is semiregular. According to Lemma \ref{semiregul groups}(ii), $\mathcal{O}_{\dis_\alpha}=\c{\dis_\alpha}$. The inclusion $\mathcal{O}_{\dis_\alpha}\leq\alpha\leq \c{\dis_\alpha}$ holds, then (i) holds.


For (ii) and (iii) we can apply Lemma \ref{central orbits} since $\alpha=\mathcal{O}_{\dis_\alpha}$.
%
%
\end{proof}

According to Corollary \ref{strucure_of_K_N}, we can apply Corollary \ref{semiregul groups_2} to central congruences of finite connected faithful idempotent left quasigroup.

\begin{corollary}\label{strucure_of_K_N_2}
Let $Q$ be a finite connected faithful idempotent left quasigroup and $\alpha$ be a central congruence. Then:
\begin{itemize}
\item[(i)] the blocks of $\alpha$ are affine latin quandles.
\item[(ii)] If $Q/\alpha$ is superfaithful (resp. superconnected, resp. $Fix(L_{[x]})=\{[x]\}$ for every $[x]\in Q/\alpha$) then $Q$ is superfaithful (resp. superconnected, resp. $Fix(L_{[x]})=\{x\}$ for every $x\in Q$). 
\end{itemize}
\end{corollary}


Let us show a converse of Corollary \ref{strucure_of_K_N_2}(ii).

\begin{lemma}\label{proj sub}
	Let $Q$ be a finite connected idempotent left quasigroup and $\alpha$ be a central congruence. If $Fix(L_x)=\{x\}$ for every $x\in Q$ then $Fix(L_{[x]})=\{[x]\}$ for every $[x]\in Q/\alpha$.
\end{lemma}


\begin{proof}
Note that $Q$ is faithful. Assume that $[x][y]=[y]$. Then $L_x L_y^{-1}\in \dis(Q)_{[y]}$. According to Corollary \ref{strucure_of_K_N} the block stabilizer is the direct product of $\dis_\alpha=\setof{L_z L_y^{-1}}{z\in [y]}$ and the stabilizer of $y$ in $\dis(Q)$. Thus, there exists $u\,\alpha\,y$ and $h\in \dis(Q)_y=\dis(Q)_u$ such that $L_x L_y^{-1}=h L_u L_y^{-1}$. Then $L_x=h L_u\in \lmlt(Q)_u$ and accordingly $xu=u$. Therefore, $x=u\,\alpha\, y$ and so $Fix(L_{[x]})=\{[x]\}$.
\end{proof}

A rack $Q$ is nilpotent if and only if $\dis(Q)$ is nilpotent and if $Q$ is a connected quandle and $|Q|=p^n$ for some prime $p$, then $Q$ is nilpotent and $\dis(Q)$ is a $p$-group. Moreover, the displacement group of finite latin quandles is solvable \cite{Stein2, CP}. In general, the relationship between nilpotency of left quasigroups and nilpotency of the displacement group in unclear. On the other hand, under further assumptions some weaker results, as the following, hold.

\begin{corollary}	\label{solvable}
Let $Q$ be a superconnected idempotent left quasigroup. If $Q$ is nilpotent of length $n$ then $\dis(Q)$ is solvable of length at most $n$.
\end{corollary}

\begin{proof}
If $Q$ is abelian, then $\dis(Q)$ is abelian. Assume that $Q$ is nilpotent of length $n$. According to Lemma \ref{O onto super} $\zeta_Q=\mathcal{O}_{\dis_{\zeta_Q}}$ and so applying Lemma \ref{central orbits} we have that $\dis^{\zeta_Q}$ is abelian. The factor $Q/\zeta_Q$ is nilpotent of length $n-1$ and by induction $\dis(Q/\zeta_Q)$ is solvable of length at most $n-1$. Thus, $\dis(Q)$ is solvable of length at most $n$.
\end{proof}

Let $Q$ be a finite left quasigroup connected by $N$. Then we have that $|N|=|x^N||N_x|=|Q||N_x|$ for every $x\in Q$. If $p$ divides $|Q|$ then clearly $p$ divides $|N|$. Under further assumptions the converse holds.

\begin{proposition}\label{divisors}
Let $Q$ be a finite idempotent superconnected left quasigroup and let $p$ be a prime. If $Q$ is nilpotent, then $p$ divides $|Q|$ if and only if $p$ divides $|\dis(Q)|$. 
In particular, if $|Q|=p^n$ then $\dis(Q)$ is a $p$-group.
\end{proposition}

\begin{proof}

Let assume that $p$ divides $|\dis(Q)|=|\dis(Q/\zeta_Q)||\dis^{\zeta_Q}|$. If $Q$ is abelian we are done, since $|Q|=|\dis(Q)|$. Assume that $Q$ is nilpotent of length $n$. By Lemma \ref{O onto super} we have $\zeta_Q=\mathcal{O}_{\dis_{\zeta_Q}}$. We need to discuss two cases:
\begin{itemize}
\item[(i)]  If $p$ divides $|\dis(Q/\zeta_Q)|$ then, by induction $p$ divides $|Q/\zeta_Q|$ and so it divides also $|Q|=|Q/\zeta_Q||[x]_{\zeta_Q}|$. 
\item[(ii)] If $p$ divides $|\dis^{\zeta_Q}|$ then, according to Lemma \ref{central orbits}, $p$ divides $|[x]|_{\zeta_Q}$ since $\zeta_Q=\mathcal{O}_{\dis_{\zeta_Q}}$. Hence, $p$ divides $|Q|=|Q/\zeta_Q||[x]_{\zeta_Q}|$.\qedhere
\end{itemize}
\end{proof}

\subsection{Central extension}

Let $Q$ be a left quasigroup, $A$ an abelian group, $f\in \aut{A}$, $g\in \End{A}$ and $\theta:Q\times Q\longrightarrow A$ be a map. The algebraic structure $E=(Q\times A, \cdot)$ where
\begin{equation}\label{central ext}
(x,a)\cdot(y,b)=(x\cdot y, g(a)+f(b)+\theta(x,y)).
\end{equation}
is a left quasigroup and it is called a {\it central extension} of $Q$ by $A$. 
The map
$$p_1:E\longrightarrow Q,\quad (x,a)\mapsto x$$
is a morphism of left quasigroups. We denote the left quasigroup $E$ by $\aff(Q,A,g,f,\theta)$. If $|Q|=1$, we can identify $Q\times A$ with $A$ and \eqref{central ext} reads
$$a\cdot b=g(a)+f(b)+c$$
for some $c\in A$. For this special case we use the notation $E=\aff(A,g,f,c)$ and we say that $E$ is an {\it affine left quasigroup} over $A$. Affine left quasigroups are abelian (they are reducts of modules).
\begin{remark}\label{remark on latin}
Note that: 
\begin{itemize}
\item[(i)] $E$ is idempotent if and only if $Q$ is idempotent, $g=1-f$ and $\theta_{x,x}=0$. 
\item[(ii)] $E$ is latin if and only if $Q$ is latin and $g\in \aut{A,+}$. In particular, if $E$ is idempotent then $E$ is latin if and only if $Q$ and $[x]_{\ker{p}}\cong \aff(A,1-f,f,0)$ are latin.
\end{itemize}
\end{remark}

Central extensions are define in the framework of Malt'sev algebraic structures \cite{comm}. In particular, every nilpotent algebraic structure with a Malt'sev term can be constructed by a chain of central extensions (in the case of quandles weaker assumptions lead to a similar result \cite[Section 7]{CP}).

\section{Malt'sev left quasigroups}\label{matlsev section}

Malt'sev classes of left quasigroups have been investigated in \cite{Maltsev_paper}. In particular Malt'sev left quasigroups are superconnected. In this section we extend some results on Malt'sev quandles and Malt'sev varieties of quandles to the setting of idempotent left quasigroups.

\begin{lemma}\label{O onto Maltsev}
Let $Q$ be a Malt'sev left quasigroup. Then $\mathcal{O}_{\dis_\alpha}=\mathcal{O}_{\dis^\alpha}=\alpha$ for every $\alpha\in Con(Q)$.
\end{lemma}

\begin{proof}
Let $\alpha\in Con(Q)$ and $\beta=\mathcal{O}_{\dis_\alpha}\leq \mathcal{O}_{\dis^\alpha}\leq \alpha$. The left quasigroup $Q$ is Malt'sev and so it has no non-trivial strongly abelian congruence, so $\dis_\alpha\neq 1$. According to Lemma \ref{factor of lambda}(iii) we have $\alpha/\beta\leq \lambda_{Q/\beta}$, and so $\alpha=\beta$ since also $Q/\beta$ is Malt'sev. Hence, $\alpha=\mathcal{O}_{\dis_\alpha}=\mathcal{O}_{\dis^\alpha}$.
\end{proof}

Using Lemma \ref{O onto Maltsev}, if $Q$ is a finite nilpotent Malt'sev left quasigroup then $\zeta_Q=\mathcal{O}_{\dis_{\zeta_Q}}$ and so Lemma \ref{central orbits} applies to $\zeta_Q$. Hence we can employ the same argument of Proposition \ref{divisors} to prove the following.

\begin{proposition}
Let $Q$ be a finite nilpotent Malt'sev left quasigroup and let $p$ be a prime. Then $p$ divides $|Q|$ if and only if $p$ divides $|\dis(Q)|$. 
In particular, if $|Q|=p^n$ then $\dis(Q)$ is a $p$-group.
\end{proposition}

Using central extensions, we can extend \cite[Theorem 2.15]{Super}. 
%


%
%


%

\begin{theorem}\label{nilpotent are latin}
Idempotent Malt'sev nilpotent left quasigroups are latin. In particular, idempotent Malt'sev abelian left quasigroups are affine latin quandles.
\end{theorem}

\begin{proof}
According to Corollary \ref{corollary for abelian}(ii), if $Q$ is an idempotent abelian left quasigroup, then $Q$ is a quandle. Hence $Q$ is latin according to \cite[Theorem 2.15]{Super} and so affine (connected abelian quandles are affine, see \cite[Section 7]{hsv}).

Assume that $Q$ is nilpotent of length $n$, then $Q$ is a central extension of $Q/\zeta_Q$. Then $Q/\zeta_Q$ is latin by induction on the nilpotency length and the blocks of $\zeta_Q$ are latin, since they are abelian. Thus, we can conclude that $Q$ is latin by Remark \ref{remark on latin}(ii).
%
%
\end{proof}

Let us show a characterization of finite nilpotent idempotent left quasigroups with a Malt'sev term.

%
%
%
%

\begin{proposition}\label{nilpotent latin}
	Let $Q$ be a finite nilpotent idempotent left quasigroup. The following are equivalent:
\begin{itemize}
\item[(i)] $Q$ is connected and $Fix(L_x)=\{x\}$ for every $x\in Q$.
\item[(ii)] $Q$ is superconnected. 
\item[(iii)] $Q$ is Malt'sev.
\item[(iv)] $Q$ is latin. 
\end{itemize}	
\end{proposition}

\begin{proof}

The implication (iv) $\Rightarrow$ (i) is true in general (latin left quasigroup are connected and Corollary \ref{P_2 and idempotent}(i) applies).

(i) $\Rightarrow$ (ii) Note that $Q$ is faithful. Let us proceed by induction on the nilpotency length. If $Q$ is abelian then $Q$ is latin by Corollary \ref{corollary for abelian}(ii) and therefore $Q$ is superconnected. Let $Q$ be nilpotent of length $n$. By Lemma \ref{proj sub}, $Q/\zeta_Q$ satisfies the hypotesis and then by induction $Q/\zeta_Q$ is superconnected. Hence, we can conclude by Corollary \ref{strucure_of_K_N_2}(ii).
%

(ii) $\Rightarrow$ (iii) According to \cite[Corollary 3.8]{Maltsev_paper}, $Q$ has a Malcev term.

(iii) $\Rightarrow$ (iv) We can apply Theorem \ref{nilpotent are latin}.
\end{proof}


The key fact in the sequel of the section is the following: idempotent semiregular left quasigroups are quandles (see Lemma \ref{semiregular idempotent are quandles}), and so in particular abelian idempotent left quasigroups are quandles (see Corollary \ref{corollary for abelian}(i)).

According to \cite[Theorem 4.21]{Maltsev_paper} the class of semiregular quandles of a Malt'sev variety of quandles is a subvariety. A direct consequence of such a theorem and the fact above is the following.

\begin{corollary}\label{maltsev semiregular}
Let $\mathcal{V}$ be a Malt'sev variety of idempotent left quasigroup. The class of semiregular left quasigroups of $\mathcal{V}$ is a subvariety of quandles of $\mathcal{V}$.
\end{corollary}

A variety $\mathcal{V}$ is said to be {\it congruence distributive} if the congruence lattice of all the algebraic structures in $\mathcal{V}$ is distributive. We can extend the characterization of distributive varieties of quandles to varieties of idempotent left quasigroups.
\begin{proposition}
Let $\mathcal{V}$ be a variety of idempotent left quasigroup. The following are equivalent:
\begin{itemize}
\item[(i)] $\mathcal{V}$ contains an abelian left quasigroup.
\item[(ii)] $\mathcal{V}$ contains an abelian quandle.
\item[(iii)] $\mathcal{V}$ contains a finite quandle.
\end{itemize}

\end{proposition}

\begin{proof}
(i) $\Rightarrow$ (ii) Idempotent abelian left quasigroups are quandles, see Corollary \ref{corollary for abelian}(ii).

(ii) $\Rightarrow$ (iii) Let $Q\in \mathcal{V}$ be an abelian quandle. According to \cite[Theorem 4.22]{Maltsev_paper} the variety generated by $Q$ contains a finite quandle.

(iii) $\Rightarrow$ (i) Let $Q\in \mathcal{V}$ be a finite quandle. Then the variety generated by $Q$ contains an abelian quandle according to \cite[Theorem 4.22]{Maltsev_paper}.
\end{proof}

\begin{corollary}\label{distributive}
Let $\mathcal{V}$ be a variety of idempotent left quasigroups. The
following are equivalent:
\begin{itemize}
\item[(i)] $\mathcal{V}$ is congruence distributive.
\item[(ii)] $\mathcal{V}$ does not contain a finite quandle. 
\end{itemize}
\end{corollary}

\bibliographystyle{amsalpha}
\bibliography{references} 

\end{document}